\documentclass{article}

\usepackage[cp1251]{inputenc}
\usepackage[russian]{babel}
\usepackage{amsmath}
\usepackage{amsfonts}
\usepackage{latexsym}
\usepackage{graphicx}
\usepackage{amsthm}

\def\Char{\operatorname{char}}

\def\A{\operatorname{A}}
\def\D{\operatorname{D}}
\def\E{\operatorname{E}}
\def\GL{\operatorname{GL}}
\def\SL{\operatorname{SL}}

\def\a{\alpha}
\def\b{\beta}
\def\g{\gamma}
\def\d{\delta}
\def\z{\zeta}
\def\l{\lambda}
\def\s{\sigma}

\def\Z{{\Bbb Z}}

\def\dynkinal#1#2#3#4{\vcenter{\vbox{\vfill
\hbox{$#1#2\cdots#3#4$}\vfill}}}

\def\dynkindfour#1#2#3#4{\vcenter{\vbox{\vfill
\hbox{$#1$}\nointerlineskip\vskip -1pt
\hbox{$\phantom{#1}#3#4$\hfil}\nointerlineskip\vskip -1pt
\hbox{$#2$}\vfill}}}

\def\dynkindeight#1#2#3#4#5#6#7#8{\vcenter{\vbox{\vfill
\hbox{$#1$}\nointerlineskip\vskip -1pt
\hbox{$\phantom{#1}#3#4#5#6#7#8$\hfil}\nointerlineskip\vskip -1pt
\hbox{$#2$}\vfill}}}

\def\dynkineeight#1#2#3#4#5#6#7#8{\vcenter{\vbox{\vfill
\hbox{$#1#3#4#5#6#7#8$}\nointerlineskip\vskip 3pt
\hbox{$\phantom{#1}\phantom{#3}#2$\hfil}\vfill}}}

\newtheorem{lemma}{Лемма}
\newtheorem{theorem}{Теорема}
\newtheorem*{theorem*}{Теорема}

\begin{document}

\begin{center}
{\bf  Орбиты векторов некоторых представлений. II}

\vspace{2mm}

И.М. Певзнер \footnote[1]{Настоящая работа выполнена при
содействии проекта РФФИ 19-01-00297}

\end{center}

\vspace{1cm}

Пусть $\Phi$~--- система корней одинаковой длины и $K$~---
произвольное поле. Далее, обозначим через $\d$  максимальный
корень $\Phi$ и положим $\Phi_0 = \{\a\in\Phi; \d\perp\a\}$, $G_0
= G_{\text{\rm sc}}(\Phi_0, K)$ и $V_1 = \langle e_\a; \angle(\a,
\d) = \pi/3\rangle$, где $e_\a$~--- элементарные корневые
элементы. В настоящей серии статей рассмотрены орбиты
действия $G_0$ на $V_1$.

Такое действие изучалось во множестве работ. Прежде всего это,
разумеется, случай $\Phi=\E_8$~--- тогда получается $56$-мерное
минимальное микровесовое представление группы типа $\E_7$.
Остальные случаи исследуются меньше, однако тоже встречаются
достаточно часто. Наряду с изучением собственно представления
$G_0$ в $V_1$, это может помочь и при исследовании представления
всей группы $G_{\text{\rm sc}}(\Phi, K)$ и соответствующей алгебры
Ли.

Настоящая статья является продолжением работы [18]. В ней были доказаны некоторые общие результаты и
разобраны случаи $\Phi=\E_l$ при $\Char K\neq 2$. В настоящей работе будут
рассмотрены случаи $\Phi=\A_l$ и $\D_l$ при $\Char K\neq 2$.

Кроме того, для случая $\Phi=\D_l$ возникла необходимость, как довольно часто бывает при изучении алгебр Ли и групп Шевалле, вычислить несколько структурных констант $N_{\a\b}$. В [4, 32] приведен метод их вычисления для систем корней одной длины и положительных базисов Шевалле. Он позволяет индуктивно находить $N_{\a\b}$. А именно, сводит их сперва, постепенно уменьшая высоту корня $\a$, к вычислению нескольких $N_{\a_i\b}$, а те, в свою очередь, сводит к нескольким $N_{\a_i\a_j}$. К сожалению, для корней большой высоты это вычисление может оказаться довольно длительным. В настоящей работе доказывается следующая теорема, позволяющая избавиться от второй части вышеописанных расчетов.

\begin{theorem*}
Пусть $\Phi$~--- система корней одной длины, в соответствующей алгебре Ли $V(\Phi)$ выбран положительный базис Шевалле, $\a_i$~--- простой корень, а $\b,\b+\a_i$~--- положительные корни. Тогда $N_{\a_i\b} = -1 \Leftrightarrow$ $i$ больше номера любого простого корня, входящего в разложение $\b$.
\end{theorem*}

\begin{center}
{\bf \S1. Основные обозначения}
\end{center}

Пусть $K$~--- произвольное поле, $\Phi$~--- система корней одной длины, $V = V(\Phi)$~---
соответствующая алгебра Ли,  а $G=G_{\text{\rm sc}}(\Phi, K)$~--- соответствующая односвязная
группа.

Как известно, в $V$ существует базис Шевалле $\{e_\a,\a\in\Phi;
h_\a,\a\in\Pi\}$, где $\Pi$~-- фундаментальная система корней. При
этом все $h_\a$ из подалгебры Картана; $[e_\a,e_{-\a}] = h_\a$;
$[h_\a,e_\b] = A_{\a\b}e_\b$, где $A_{\a\b} =
2(\a,\b)/(\a,\a)\in\Z$~--- числа Картана; $[e_\a,e_\b] =
N_{\a\b}e_{\a+\b}$ при $\a+\b\in\Phi$ и $[e_\a,e_\b]=0$ при
$\a+\b\notin\Phi$ и $\b\neq -\a$, где $N_{\a\b} = \pm1$~---
структурные константы. Коэффициент в разложении вектора $x\in V$
по этому базису при $e_\a$ обозначим $x^\a$, а соответствующий
элемент из подалгебры Картана обозначим $x^h$; тогда $x =
\sum_{\a\in\Phi} x^\a e_\a + x^h$. 

Далее, в группе $G$ выделяются элементарные корневые элементы
$x_\a(a)$, $\a\in\Phi, a\in K$ и $X_\a = \langle x_\a(a); a\in
K\rangle$~--- элементарные корневые подгруппы. В работе будут
использоваться формулы для действия $x_\a(a)$ на базисе Шевалле.
Они перечислены, например, в [34]. Нам понадобятся следующие
равенства: $x_\a (a) e_\b = e_\b$ при $\angle(\a,\b)<2\pi/3$,
$x_\a (a) e_\b = e_\b + N_{\a\b}  a e_{\a+\b}$ при
$\angle(\a,\b)=2\pi/3$, $x_\a (a) e_{-\a} = e_{-\a} + a h_\a - a^2
e_\a$ и $x_\a (a) h_\b = h_\b - A_{\b\a} a e_\a$; мы будем ими
пользоваться без дополнительных ссылок.

Обозначим через $\d$ максимальный корень системы $\Phi$.
Экстраспециальным унипотентным радикалом называется подгруппа
$U_\d = \langle X_\a; \angle(\a,\d)<\pi/2\rangle$. Подробнее об
этом радикале говорится, например, в [8, 16].

Пусть $\a\in\Phi$~--- некоторый корень. Разобьем все корни из
$\Phi$ на пять классов в зависимости от их расположения
относительно корня $\a$: $\Phi_2(\a) = \{\a\}$, $\Phi_1(\a) =
\{\b; \angle(\b,\a) = \pi/3\}$, $\Phi_0(\a) = \{\b; \angle(\b,\a)
= \pi/2\}$, $\Phi_{-1}(\a) = \{\b; \angle(\b,\a) = 2\pi/3\}$,
$\Phi_{-2}(\a) = \{-\a\}$. Другими словами, $\b$ принадлежит
$\Phi_i(\a)$ тогда и только тогда, когда скалярное произведение
$\b$ и $\a$ равно $i/2$. Простейшие свойства этого разбиения
приведены в [16]. Часто нас будет интересовать случай $\a=\d$; для
краткости, аргумент у $\Phi_i(\d)$ будем опускать. В этих
обозначениях экстраспециальный радикал $U_{\d} = \langle X_{\a};
\a\in\Phi_2\cup\Phi_1\rangle$.

\begin{center}
{\bf \S2. Структурные константы}
\end{center}

В отличие от статьи [18], в настоящей работе нам понадобится
посчитать несколько структурных констант. В начале мы кратко
изложим метод, описанный в [4], потом упростим и
ускорим его, и в конце параграфа посчитаем требуемые константы.

1. Структурные константы $N_{\a\b}$, $\a,\b\in\Phi$, как уже
отмечалось, определяются равенством $[e_\a,e_\b] =
N_{\a\b}e_{\a+\b}$. По определению базиса Шевалле все $N_{\a\b}$
являются целыми числами; $N_{\a\b}=0$ тогда и только тогда, когда
$\a+\b\notin\Phi$. Для систем корней одинаковой длины при этом
$N_{\a\b}=\pm1$ при $\a+\b\in\Phi$.

Далее, существуют так называемые положительные базисы Шевалле, в
которых $N_{\a\b}>0$ для всех экстра-специальных пар (см. [4, 32]).
Для систем, в которых все корни имеют одинаковую длину, это
условие означает в точности, что $N_{\a_i\b}=1$ каждый раз, как
$\a_i+\b\in\Phi^+$ обладает тем свойством, что если $\a_j+\g =
\a_i+\b$ для какого-то простого корня $\a_j$ и какого-то
положительного корня $\g$, то $j>i$. Мы всегда полагаем, что
выбранный нами базис Шевалле положителен.

Для структурных констант, как известно (см., например, [32, 34]), выполняются
следующие свойства:

\begin{enumerate}
\item[N1] $N_{\a\b} = N_{-\b,-\a} = -N_{-\a,-\b} = -N_{\b\a}$;

\item[N2] $\frac{N_{\a\b}}{(\g,\g)} = \frac{N_{\b\g}}{(\a,\a)} =
\frac{N_{\g\a}}{(\b,\b)}$, если $\a + \b + \g = 0$;

\item[N3] $\frac{N_{\a\b}N_{\g\z}}{(\a+\b,\a+\b)} +
\frac{N_{\b\g}N_{\a\z}}{(\b+\g,\b+\g)} +
\frac{N_{\g\a}N_{\b\z}}{(\g+\a,\g+\a)} = 0$, если $\a + \b + \g +
\z = 0$.
\end{enumerate}

Для систем с одинаковыми длинами корней формулы N2 и N3 можно
упростить:

\begin{enumerate}
\item[N2'] $N_{\a\b} = N_{\b\g} = N_{\g\a}$, если $\a + \b + \g =
0$;

\item[N3'] $N_{\a\b}N_{\g\z} + N_{\b\g}N_{\a\z} + N_{\g\a}N_{\b\z}
= 0$, если $\a + \b + \g + \z = 0$.
\end{enumerate}

На самом деле, в последней формуле одно слагаемое точно будет
нулевым, поэтому она эквивалентна части уравнения 2-цикла

\begin{enumerate}
\item[N3''] $N_{\b\g}N_{\a,\b+\g} = N_{\a+\b,\g}N_{\a\b}$.
\end{enumerate}

Свойство положительности базиса Шевалле можно записать следующим
образом:

\begin{enumerate}
\item[N4] $N_{\a_i\b}=1$, если $\a_i$~--- простой корень с
наименьшим индексом, который можно вычесть из корня $\a_i+\b$.
\end{enumerate}

Здесь и всюду далее в этом параграфе под "можно вычесть один
корень из другого"{} подразумевается, что их разность является
корнем. Как известно, это равносильно тому, что скалярное произведение этих корней равно $1/2$.

2. В работе [4], следуя работе Титса [32], приведен индуктивный способ
вычисления $N_{\pm\a_i\b}$. Однако, как мы покажем в этом
параграфе, эти константы можно найти сразу, без длительных
подсчетов. Сперва докажем вспомогательную лемму.

\begin{lemma}
Пусть $\g\in\Phi^+$~--- некоторый положительный корень и $\a_j$~---
простой корень с наименьшим индексом, который можно вычесть из
$\g$. Далее, пусть $\a_i$~--- другой простой корень, который можно
вычесть из $\g$, и $\g \neq \a_i+\a_j$. Тогда либо $\a_j$ является
простым корнем с наименьшим индексом, который можно вычесть из
$\g-\a_i$, либо $\Phi=\E_8$, $\g=\dynkineeight23454321$, $j=2$ и
$i=3$.
\end{lemma}
\begin{proof}
Так как $\a_j$ вычитается из $\g$, то $(\g,\a_j)=1/2$. Аналогично
$(\g,\a_i)=1/2$, следовательно $(\g,\a_i+\a_j)=1$. Если бы
$\a_i+\a_j$ являлось бы корнем, то $\g=\a_i+\a_j$, что
противоречит условию. Значит $\a_i+\a_j$ не корень и,
следовательно, $(\a_i,\a_j)=0$ (то есть $\a_i$ и $\a_j$ не
соседние простые корни). Поэтому $(\g-\a_i,\a_j)=1/2$ и $\a_j$
вычитается из $\g-\a_i$.

Далее, напомним хорошо известное и совсем простое утверждение: для
произвольного корня $\g$ в его разложении на простые корни
удвоенный коэффициент при любом простом корне $\a$ или равен сумме
коэффициентов при всех соседних с ним в диаграмме Дынкина простых
корнях, или отличается от нее на $1$. При этом если они равны, то
$(\a,\g) = 0$; если удвоенный коэффициент больше на $1$, то
$(\a,\g) = 1/2$ и $\a$ можно вычесть из $\g$ (или $\g=\a$ и $(\a,\g)=1$); если
удвоенный коэффициент меньше на $1$, то $(\a,\g) = -1/2$ и $\a$ можно прибавить к $\g$ (или $\g=-\a$ и $(\a,\g)=-1$). Это сразу следует из того, что
скалярный квадрат $\a$ равен $1$, скалярное произведение $\a$ на
соседний корень равно $1/2$, а на не соседний~--- $0$.

Предположим, что существует простой корень с меньшим индексом, чем
$j$, который тоже можно вычесть из $\g-\a_i$; пусть это $\a_k$,
$k<j$. Так как $\a_k$ не вычиталось из $\g$, то, в силу
вышесказанного, это означает, что простые корни $\a_k$ и $\a_i$
соседние. Таким образом, $k<j<i$, корни $\a_k$ и $\a_i$ соседние,
а $\a_j$ и $\a_i$~--- нет. В системах корней $\Phi=\A_l$ соседние
простые корни всегда имеют соседние номера, и сразу получается
противоречие. Если $\Phi=\D_l$, то единственные соседние корни с
не соседними номерами, это $\a_1$ и $\a_3$. Но тогда $j=2$ и корни
$\a_2$ и $\a_3$~--- тоже соседние, что противоречит ранее
сказанному.

Пусть $\Phi=\E_l$. Аналогичными рассуждениями получаем, что
единственный возможный случай~--- это $k=1$, $j=2$ и $i=3$. Тогда
в разложении $\g$ коэффициент при $\a_1$ может быть равен $1$ или
$2$. Предположим, что он равен $1$. Тогда, так как $(\g,\a_1)=0$,
коэффициент при $\a_3$ равен $2$. Так как $\a_3$ вычитается из
$\g$, то коэффициент при $\a_4$ тоже должен равняться $2$ (как мы
уже упоминали, удвоенный коэффициент при $\a_3$ должен равняться
сумме коэффициентов при $\a_1$ и $\a_4$ плюс $1$). Но тогда
коэффициент при $\a_2$ должен равняться $1$ и $\a_2$ не будет
вычитаться из $\g$~--- противоречие.

Осталось рассмотреть случай, когда в разложении $\g$ коэффициент
при $\a_1$ равен $2$. Тогда, аналогично, коэффициент при $\a_3$
должен равняться $4$, при $\a_4$~--- $5$, а при $\a_2$~--- $3$. В
принципе, то что такой корень существует всего один, можно
убедиться и непосредственным перебором, но мы в этом убедимся,
используя все то же утверждение. Как мы видим (и как уже
выводилось из условия), из корня $\g$ можно вычесть $\a_2$ и
$\a_3$. Поэтому удвоенный коэффициент при $\a_4$ должен быть на
$1$ меньше суммы коэффициентов при соседних корнях, откуда
коэффициент при $\a_5$ должен равняться $4$. Далее, корень $\a_4$
можно к $\g$ прибавить, а можно, после вычитания $\a_2$ и $\a_3$,
и вычесть. Это означает, что удвоенный коэффициент при $\a_5$
должен равняться сумме коэффициентов при соседних корнях, откуда
коэффициент при $\a_6$ должен равняться $3$. Аналогичными
рассуждениями получаем, что коэффициент при $\a_7$ равен $2$, а
при $\a_8$~--- $1$, что и требовалось.
\end{proof}

\begin{theorem}
Пусть $\a_i\in\Pi$ и $\b,\b+\a_i\in\Phi^+$. Тогда $N_{\a_i\b} = -1
\Leftrightarrow$ $i$ больше номера любого простого корня,
входящего в разложение $\b$.
\end{theorem}
\begin{proof}
Прежде всего отметим, что если $\b=\a_j$~--- простой корень, то
утверждение теоремы выполняется. В самом деле, если $i<j$, то
$N_{\a_i\a_j} = 1$ по N4, а если $i>j$, то $N_{\a_i\a_j} = -
N_{\a_j\a_i}$ по N1, что, по N4, равно $-1$. Поэтому можно
считать, что $\b$ не является простым корнем. Далее, пусть
$\a_j$~--- простой корень с наименьшим индексом, который можно
вычесть из корня $\a_i+\b$. Если $j=i$, то, по свойству N4,
$N_{\a_i\b} = 1$ и утверждение теоремы выполняется. Пусть $j\neq
i$. Если в свойство N3'' подставить $\a = \a_i$, $\b = \b - \a_j$
и $\g = \a_j$, то получится равенство $N_{\b-\a_j,\a_j}
N_{\a_i\b} = N_{\a_i+\b-\a_j,\a_j} N_{\a_i,\b - \a_j}$. Отсюда
$N_{\a_i\b} = N_{\a_i+\b-\a_j,\a_j} N_{\a_i,\b - \a_j}
N_{\b-\a_j,\a_j}$, что, по свойству N1, равно
$N_{\a_j,\a_i+\b-\a_j} N_{\a_i,\b - \a_j} N_{\a_j,\b-\a_j}$. По
определению $\a_j$ и N4 первый сомножитель равен $1$.

Рассмотрим последний сомножитель, $N_{\a_j,\b-\a_j}$. Заметим, что
если в предыдущую лемму подставить $\g = \b+\a_i$, то ее условие
будет выполняться. Это означает, что, либо $\a_j$ является простым
корнем с наименьшим индексом, который можно вычесть из $\b$, либо
$\Phi=\E_8$, $\g = \b+\a_i = \dynkineeight23454321$, $j=2$ и
$i=3$. Первый из этих случаев, по N4, означает, что
$N_{\a_j,\b-\a_j} = 1$, откуда $N_{\a_i\b} = N_{\a_i,\b - \a_j}$.
Таким образом, мы можем спокойно из корня $\b$ вычитать простой
корень с наименьшим индексом, который из него вычитается~---
структурная константа при этом не меняется. Эта процедура
остановится либо когда мы придем к вышеуказанному исключительному
случаю, либо к случаю простого корня $\b=\a_j$, либо к случаю,
когда $\a_j=\a_i$. Отметим также, что при переходе от $\b$ к
$\b-a_j$ условие из формулировки теоремы не меняется: если $i$
было больше номера любого простого корня, входящего в разложение
$\b$, то это верно и для $\b-\a_j$, и наоборот. Поэтому во втором
и третьем случае теорема выполняется.

Заметим, что $N_{\a_j,\b-\a_j}$ при $\b = \dynkineeight23354321$ и $j=2$
также равен $1$. В самом деле, для этого достаточно к этой
структурной константе применить только что полученный
результат (то есть подставить в первоначальные расуждения $\b = \dynkineeight22354321$ и $i=2$)~--- то что мы не попадем в исключительный случай,
очевидно. Это означает, что исключительный случай можно было и не
исключать, так как в нем тоже $N_{\a_i\b} = N_{\a_i,\b - \a_j}$.
Поэтому утверждение теоремы всегда верно.
\end{proof}

3. Перейдем к подсчету произведений структурных констант, которые нам
понадобятся в дальнейшем. У нас будут два отдельных подсчета для
систем $\Phi=\D_4$ и $\Phi=\D_l, l>4$. В рассуждениях мы будем
использовать свойства N1, N2', N4 и результат теоремы без
подробных записей, а N3'' придется расписывать поподробнее, в силу
неоднозначности выбора корней и порядка использования этого
свойства. Отметим, что N4, разумеется, является частным случаем теоремы, но где
это возможно, для простоты, мы будем ссылаться именно на N4.
Начнем со случая $\Phi=\D_4$.

\begin{lemma} Пусть $\Phi=\D_4$, и положим $\l=\dynkindfour0010$,
$\rho=\dynkindfour0111$, $\s=\dynkindfour1011$,
$\tau=\dynkindfour1110$. Тогда
\begin{enumerate}
\item $N_{\l-\d,\d}N_{\l-\d,\rho}N_{\s-\d,\d}N_{\s-\d,\tau} = 1$;

\item $N_{\l-\d,\d}N_{\l-\d,\s}N_{\rho-\d,\d}N_{\rho-\d,\tau} =
1$.
\end{enumerate}
\end{lemma}
\begin{proof}
$N_{\l-\d,\d} \stackrel{N2'}{=} N_{-\l,\l-\d} \stackrel{N1}{=}
-N_{\l,\d-\l}$. Так как $\l+(\d-\l)=\d$~--- максимальный корень, а
$\l=\a_3$~--- единственный простой корень, который из него
вычитается, то по N4 это выражение равно $-1$.

Далее, $N_{\l-\d,\rho} \stackrel{N1}{=} -N_{\rho,\l-\d}
\stackrel{N2'}{=} -N_{\d-\l-\rho,\rho}$. Так как $\d-\l-\rho =
\a_1$, то по N4 это выражение снова равно $-1$.

Считаем третью константу: $N_{\s-\d,\d} \stackrel{N2'}{=}
N_{-\s,\s-\d} \stackrel{N1}{=} N_{\d-\s,\s}$. Заметим, что $\d-\s
= \a_2 + \a_3$, то есть $N_{\d-\s,\s} = N_{\a_2+\a_3,\s}$.
Подставим в N3'' $\a=\a_3$, $\b=\a_2$ и $\g=\s$. Тогда получим
$N_{\a_2\s}N_{\a_3,\a_2+\s} = N_{\a_3+\a_2,\s}N_{\a_3\a_2}$. При
этом $N_{\a_2\s} = 1$ по теореме, $N_{\a_3,\a_2+\s} = 1$ по
свойству N4, а $N_{\a_3\a_2} = -1$ снова по теореме. Поэтому
третья константа тоже получается равна $-1$.

Четвертая константа: $N_{\s-\d,\tau} \stackrel{N2'}{=}
N_{\d-\s-\tau,\s-\d} \stackrel{N1}{=} -N_{\s+\tau-\d,\d-\s}$. Так
как $\s+\tau-\d=\a_1$, то по N4 это выражение снова равно $-1$.
Таким образом, все четыре константы в левой части первого пункта
равны $-1$, а их произведение, соответственно, $1$. Это завершает
доказательство первого пункта.

Далее, $N_{\l-\d,\s} \stackrel{N1}{=} -N_{\s,\l-\d}
\stackrel{N2'}{=} -N_{\d-\l-\s,\s}$. Так как $\d-\l-\s = \a_2$, то
это выражение есть $-N_{\a_2\s}$, что по теореме, как мы уже
говорили, равно $-1$.

Затем, $N_{\rho-\d,\d} \stackrel{N2'}{=} N_{-\rho,\rho-\d}
\stackrel{N1}{=} N_{\d-\rho,\rho}$. Заметим, что $\d-\rho = \a_1 +
\a_3$, то есть $N_{\d-\rho,\rho} = N_{\a_1+\a_3,\rho}$. Подставим
в N3'' $\a=\a_3$, $\b=\a_1$ и $\g=\rho$. Тогда получим
$N_{\a_1\rho}N_{\a_3,\a_1+\rho} =
N_{\a_3+\a_1,\rho}N_{\a_3\a_1}$. При этом $N_{\a_1\rho} = 1$ по
свойству N4, $N_{\a_3,\a_1+\rho} = 1$ тоже по свойству N4, а
$N_{\a_3\a_1} = -1$ по теореме. Поэтому эта константа также
получается равна $-1$.

Наконец, $N_{\rho-\d,\tau} \stackrel{N2'}{=}
N_{\d-\rho-\tau,\rho-\d} \stackrel{N1}{=}
-N_{\rho+\tau-\d,\d-\rho}$. Так как $\rho+\tau-\d=\a_2$, то по
теореме это выражение снова равно $-1$. Таким образом, все четыре
константы в левой части второго пункта также равны $-1$, а их
произведение, соответственно, $1$. Это завершает доказательство
второго пункта.
\end{proof}

Отметим, что пункты 1 и 2 этой леммы получаются перестановкой
выбранных корней. Более того, на самом деле для случая $\Phi=\D_4$
для любых четырех попарно ортогональных корней, образующих угол
$\pi/3$ с корнем $\d$, аналогичное произведение будет равно $1$.
Однако для общего случая $\Phi=\D_l$ это не так, знак произведения
меняется в зависимости от выбора корней. Посчитаем этот знак для
одной из четверок:

\begin{lemma} Пусть $\Phi=\D_l$, $l>4$, и положим $\l=\dynkindeight000\dots 0010$,
$\rho=\dynkindeight112\dots2210$, $\s=\dynkindeight000\dots0111$,
$\tau=\dynkindeight112\dots2111$. Тогда
$N_{\l-\d,\d}N_{\l-\d,\rho}N_{\s-\d,\d}N_{\s-\d,\tau} = 1$.
\end{lemma}
\begin{proof}
Первая из структурных констант считается также, как и в прошлой
лемме: $N_{\l-\d,\d} \stackrel{N2'}{=} N_{-\l,\l-\d}
\stackrel{N1}{=} -N_{\l,\d-\l}$. Так как $\l+(\d-\l)=\d$~---
максимальный корень, а $\l$~--- единственный простой корень,
который из него вычитается, то по N4 это выражение равно $-1$.

Далее, $N_{\l-\d,\rho} \stackrel{N2'}{=} N_{\rho,\d-\l-\rho}
\stackrel{N1}{=} -N_{\d-\l-\rho,\rho}$. Так как $\d-\l-\rho =
\dynkindeight000\dots 0001 = \a_l$, а $\rho =
\dynkindeight112\dots2210$, то, по теореме, $N_{\d-\l-\rho,\rho} =
-1$, а искомая константа, соответственно, равна $1$.

Считаем третью константу: $N_{\s-\d,\d} \stackrel{N2'}{=}
N_{-\s,\s-\d} \stackrel{N1}{=} -N_{\s,\d-\s}$. По условию, $\s =
\a_{n-2}+\a_{n-1}+\a_n$. Подставим в N3'' $\a=\a_{n-2}+\a_{n-1}$,
$\b=\a_n$ и $\g=\d-\s$. Тогда получим:
$N_{\a_n,\d-\s}N_{\a_{n-2}+\a_{n-1},\a_n+\d-\s} =
N_{\s,\d-\s}N_{\a_{n-2}+\a_{n-1},\a_n}$. При этом $N_{\a_n,\d-\s}
= -1$ по теореме и $N_{\a_{n-2}+\a_{n-1},\a_n} \stackrel{N1}{=}
-N_{\a_n,\a_{n-2}+\a_{n-1}} = 1$ снова по теореме. Чтобы посчитать
оставшийся коэффициент, $N_{\a_{n-2}+\a_{n-1},\a_n+\d-\s}$, еще
раз воспользуемся N3'': подставим $\a=\a_{n-1}$, $\b=\a_{n-2}$ и
$\g=\a_n+\d-\s$. Тогда получим
$N_{\a_{n-2},\a_n+\d-\s}N_{\a_{n-1},\a_{n-2}+\a_n+\d-\s} =
N_{\a_{n-1}+\a_{n-2},\a_n+\d-\s}N_{\a_{n-1},\a_{n-2}}$. Здесь
$N_{\a_{n-2},\a_n+\d-\s}$ и $N_{\a_{n-1},\a_{n-2}+\a_n+\d-\s}$
равны $1$ по теореме, а $N_{\a_{n-1},\a_{n-2}}$, тоже по теореме,
равно $-1$. Таким образом, $N_{\a_{n-1}+\a_{n-2},\a_n+\d-\s}=-1$,
откуда $N_{\s,\d-\s}=1$ и, наконец, $N_{\s-\d,\d}=-1$.

Считаем четвертую константу: $N_{\s-\d,\tau} \stackrel{N1}{=}
-N_{\d-\s,-\tau}\stackrel{N2'}{=}-N_{\s+\tau-\d,\d-\s}$. Заметим,
что $\s+\tau-\d=\a_n$, поэтому $N_{\s+\tau-\d,\d-\s}=-1$ по
теореме. Следовательно, $N_{\s-\d,\tau}=1$. Осталось собрать все
четыре найденные константы:
$N_{\l-\d,\d}N_{\l-\d,\rho}N_{\s-\d,\d}N_{\s-\d,\tau} =
(-1)\cdot1\cdot(-1)\cdot1=1$, что и требовалось.

\end{proof}

\begin{center}
{\bf \S3. Необходимые результаты из работы [18]}
\end{center}

4. Так как настоящая работа является переносом результатов статьи
[18] со случаев $\Phi=\E_6$, $\E_7$ и $\E_8$ на случаи $\A_l$ и $\D_l$,
стоит напомнить основные моменты той статьи. Мы рассматриваем
орбиты действия группы $G_0 = G_{\text{\rm sc}}(\Phi_0, K)$ на
пространстве $V_1 = \langle e_\a; \a\in\Phi_1\rangle$. Пусть $x\in
V_1$. В начале мы доказали, что можно считать, что $x=x^\l e_\l +
x^{\d-\l} e_{\d-\l} + x^\mu e_\mu + x^\nu e_\nu + x^\xi e_\xi$;
здесь $\l, \mu, \nu, \xi \in \Phi_1$~--- некоторые попарно
ортогональные корни и $\l+\mu+\nu+\xi = 2\d$. Далее мы доказывали,
что для систем $\Phi=\E_6$, $\E_7$ и $\E_8$ корни $\l$, $\mu$, $\nu$
и $\xi$ могут быть выбраны произвольно (в рассматриваемых в
настоящей статье случаях это, как мы увидим, не совсем так). После
этого мы заметили, что любой $x\in V_1$ приводится к одному из
следующих типов.

\begin{enumerate}
\item[I.] $x=0$.

\item[II.] $x=x^\l e_\l$; $x^\l\neq 0$. Есть для всех систем
$\Phi$, кроме $\A_1$.

\item[III.] $x=x^\l e_\l + x^\mu e_\mu$; $x^\l, x^\mu\neq 0$. Есть
для всех систем $\Phi$, кроме $\A_1$ и $\A_2$.

\item[IV.] $x=x^\l e_\l + x^\mu e_\mu + x^\nu e_\nu$; $x^\l,
x^\mu, x^\nu\neq 0$. Есть для всех систем $\Phi$, кроме $\A_l$.

\item[V.] $x=x^\l e_\l + x^\mu e_\mu + x^\nu e_\nu +x^\xi e_\xi$;
$x^\l, x^\mu, x^\nu, x^\xi\neq 0$. Есть для всех систем $\Phi$,
кроме $\A_l$.

\item[VI.] $x=x^\l e_\l + x^{\d-\l} e_{\d-\l}$; $x^\l,
x^{\d-\l}\neq 0$. Есть для всех систем $\Phi$, кроме $\A_1$.

\item[VII.] $x=x^\l e_\l + x^{\d-\l} e_{\d-\l} + x^\mu e_\mu$;
$x^\l, x^{\d-\l}, x^\mu\neq 0$. Есть для всех систем $\Phi$, кроме
$\A_1$ и $\A_2$.

\item[VIII.] $x=x^\l e_\l + x^{\d-\l} e_{\d-\l} + x^\mu e_\mu +
x^\nu e_\nu$; $x^\l, x^{\d-\l}, x^\mu, x^\nu\neq 0$. Есть для всех
систем $\Phi$, кроме $\A_l$.

\item[IX.] $x=x^\l e_\l + x^{\d-\l} e_{\d-\l} + x^\mu e_\mu +
x^\nu e_\nu +x^\xi e_\xi$; $x^\l, x^{\d-\l}, x^\mu,
x^\nu,x^\xi\neq 0$. Есть для всех систем $\Phi$, кроме $\A_l$.
\end{enumerate}

От VII и VIII случаев мы избавились сразу же. Далее, от VI случая
мы избавлялись при $\Phi\neq \A_l$ и $|K|>2$, а от IX~--- при
$\Char K\neq 2$. Кроме того, мы доказали следующую несложную
лемму:

\begin{lemma}[Лемма 1 из {[18]}]
Пусть
$w_{\a}(a)=x_{-\a}(-a^2+a)x_{\a}(-\frac1a)x_{-\a}(a-1)x_{\a}(1)$
при $\a\in\Phi_0$ и $a\in K^*$. Тогда $w_\a(a)e_\b = \frac1a e_\b$
при $\angle(\a,\b) = \pi/3$; $w_\a(a)e_\b = e_\b$ при
$\angle(\a,\b) = \pi/2$ и $w_\a(a)e_\b = a e_\b$ при
$\angle(\a,\b) = 2\pi/3$.
\end{lemma}

С помощью этой леммы мы избавлялись от большинства коэффициентов в
оставшихся случаях и приходили к искомому списку орбит.

5. Далее мы доказывали, что все описанные случаи дают разные
орбиты. Во-первых, была доказана следующая лемма:

\begin{lemma}[Лемма 3 из {[18]}]
Пусть $\Phi\neq \A_l$, $\Char K\neq 2$ и $x = \sum_{\a\in\Phi_1}
x^\a e_\a$. Тогда существует и единственен корневой элемент $y =
\sum_{\a\in\Phi} y^\a e_\a + y^h$, в котором $y^\a=x^\a$ при
$\a\in\Phi_1$, $y^\d = 1$ и $y^h\in\langle h_\a;
\a\in\Phi_0\cap\Pi\rangle$.
\end{lemma}

Во-вторых, мы вводили следующие определения: элемент $x$
называется темным, если угол между соответствующим ему $y$ и
$e_\d$ равен $\pi$ (иначе говоря, $y^{-\d}\neq 0$). Далее, $x$
называется светящимся, если этот угол равен $2\pi/3$; блестящим,
если угол равен $\pi/2$ и сингулярным, если он равен $\pi/3$. Как
несложно видеть, последнее условие равносильно тому, что $x$
корневой элемент. Названия взяты из работы [29] для случая $\Phi =
\E_8$. Их определения в [29] другие, но, на самом деле,
равносильные.

Затем мы заметили, что при умножении на $g\in G_0$ элементу $gx$
соответствует корневой элемент $gy$. При этом угол между $y$ и
$e_\d$ при таком умножении также не меняется, значит определения
темного, светящегося, блестящего и сингулярного элементов можно
расширить на орбиты. Более того, если $x$ темный, то коэффициент
$y^{-\d}$ также не меняется при умножении на $g\in G_0$, то есть
тоже является инвариантом. Наконец, мы проверяли, что в случае II
вектор $x$ получается сингулярным, в случае III~--- блестящим, в
случае IV~--- светящимся, а в случае V~--- темным. При этом в
случае V инвариант $y^{-\d}$ равнялся произведению нескольких
структурных констант и $x^\l x^\mu x^\nu x^\xi$, поэтому
произведение всех коэффициентов $x^\l x^\mu x^\nu x^\xi$ также
постоянно. Для $\Phi=\E_l$, рассматриваемого в статье [18], этого
хватало для классификации. Для рассматриваемых в настоящей работе
$\Phi=\A_l$ и $\D_l$ это также будет очень полезным, но, к
сожалению, для полной классификации этого не хватит.

\begin{center}
{\bf \S4. Случай $\Phi=\A_l$}
\end{center}

6. Вернемся к уже упомянутому ранее вопросу существования и
произвольности выбора четырех попарно ортогональных корней $\l$,
$\mu$, $\nu$ и $\xi$ для разных систем $\Phi$. В рассматриваемом в
этом разделе случае $\Phi=\A_l$ может быть только два ортогональных
корня, так что корни $\nu$ и $\xi$ точно отсутствуют. Мы
перечислим орбиты этих корней под действием группы Вейля системы
$\Phi_0$; так как корни выбирались в теореме по порядку, то
порядок нам тоже важен. Таким образом, для определения корня $\l$
мы можем пользоваться всей группой Вейля, а для корня $\mu$~---
только подгруппой, оставляющей уже выбранный $\l$ неподвижным.

\begin{itemize}
\item Пусть $\Phi = \A_1$. В этом случае $\Phi_1$ пустое.

\item Пусть $\Phi = \A_2$. В этом случае $\Phi_1$ состоит из двух
корней с углом $2\pi/3$ между ними, а $\Phi_0$ пустое. Поэтому у
корня $\l$ получается две орбиты по одному корню в каждой, а
корней $\mu$, $\nu$ и $\xi$ в этом случае нет. Для удобства
обозначим корень $10$ через $\rho$. Тогда $\l=\rho$ или $\d-\rho$.

\item Пусть $\Phi = \A_l$, $l>2$. В этом случае $\Phi_1$ состоит из
корней, имеющих 1 при первом фундаментальном корне и 0 при
последнем, или наоборот. Таким образом, $\Phi_1$ распадается на
две части (по $l-1$ корню в каждой), причем корни из одной части
имеют угол $\pi/3$ между собой, а корни из разных частей~--- угол
$\pi/2$ или $2\pi/3$. В этом случае у корня $\l$ получается две
орбиты, у $\mu$ одна орбита (если $\l$ в одной части, то $\mu$
точно в другой), а корней $\nu$ и $\xi$ в этом случае тоже нет.
Подобно прошлому случаю, обозначим корень $\dynkinal1000$ через
$\rho$, а корень $\dynkinal0001$ через $\s$. Тогда можно считать,
что $\l=\rho$ и $\mu=\s$, или наоборот, $\l=\s$ и $\mu=\rho$.
\end{itemize}

7. Как известно, для случая $\Phi=\A_l$ элементы $V$ можно
представить в виде матриц размера $(l+1)\times (l+1)$. В
естественной нумерации максимальному корню $\d$ соответствует
единица в правом верхнем углу. При этом элементы из $V_1$ имеют
вид

$$\left(\begin{array}{ccccc}
0      & a_1    & \ldots & a_{l-1}& 0      \\
0      & 0      & \ldots & 0      & b_1    \\
\vdots & \vdots & \ddots & \vdots & \vdots \\
0      & 0      & \ldots & 0      & b_{l-1}\\
0      & 0      & \ldots & 0      & 0
\end{array}\right),
$$

то есть им соответствует пара из ковектора и вектора $((a_1,\dots,
a_{l-1}),\\
(b_1,\dots, b_{l-1})^T)$. Каждому корню соответствует своя
координата в векторе или ковекторе. Если двум различным корням
соответствуют координаты в векторе (или ковекторе), то угол между
ними равен $\pi/3$; если одному соответствует координата в
векторе, а другому в ковекторе, причем эти координаты с различным
номером, то угол между ними равен $\pi/2$; если же эти координаты
окажутся с одним номером, то угол будет равен $2\pi/3$.

Далее, $g\in G_0$ представляется в виде

$$\left(\begin{array}{ccccc}
1      & 0      & \ldots & 0      & 0      \\
0      & *      & \ldots & *      & 0      \\
\vdots & \vdots & \ddots & \vdots & \vdots \\
0      & *      & \ldots & *      & 0      \\
0      & 0      & \ldots & 0      & 1
\end{array}\right),
$$
где центральная подматрица из $\SL(l-1,K)$; обозначим эту подматрицу $g'$. При этом действие $G_0$ на $V_1$
соответствует сопряжению. Как несложно видеть, при таком действии
пара из ковектора и вектора $(u^T,v)$ переходит в
$(u^T(g')^{-1},g'v)$.

На таком языке, разумеется, можно разобрать целиком весь случай
$\Phi=\A_l$. Однако там, где это возможно и не приводит к большим
трудностям, мы сохраним взгляд и терминологию из предыдущей статьи
(и следующего параграфа).

8. Теперь посмотрим, как можно упростить перечень типов из
пункта~4 с помощью леммы~4 для систем корней типа $\A_l$:

\begin{itemize}
\item Пусть $\Phi = \A_1$. Тогда $\Phi_1$ пустое, $V_1$ состоит из
одного нуля и, соответственно, есть ровно одна нулевая орбита.

\item Пусть $\Phi = \A_2$. Тогда $\Phi_1$ состоит из двух векторов
$\rho$ и $\d-\rho$, $\Phi_0$ пустое, значит все векторы $x^\rho
e_\rho + x^{\d-\rho} e_{\d-\rho}$, при $x^\rho, x^{\d-\rho}\in K$,
образуют орбиты из одного элемента.

\item Пусть $\Phi = \A_3$. Тогда $\Phi_1$ состоит из четырех
векторов $\rho$, $\s$, $\d-\rho$ и $\d-\s$, а $\Phi_0 =
\{\d-\rho-\s, \rho+\s-\d\} = \A_1$. Из 9 случаев, перечисленных в
пункте~4, у нас остались I, II, III и VI типы:

\begin{enumerate}
\item[I.] $x=0$.

\item[II.] $x=x^\l e_\l$; $x^\l\neq 0$.

\item[III.] $x=x^\l e_\l + x^\mu e_\mu$; $x^\l, x^\mu\neq 0$.

\item[VI.] $x=x^\l e_\l + x^{\d-\l} e_{\d-\l}$; $x^\l,
x^{\d-\l}\neq 0$.
\end{enumerate}

В случае I, разумеется, остается 0. В случае II, по лемме и
описанным ранее особенностям выбора корней $\l$ в системах $\A_l$,
все элементы приводятся к одному из двух подтипов $x=e_\rho$ и
$x=e_\s$. В случае III по лемме получается избавиться лишь от
одного коэффициента, и получается серия типов $x=e_\rho + x^\s
e_\s$, $x^\s\neq0$. В случае VI~--- тоже лишь от одного
коэффициента, и получается серия типов $x=e_\rho + x^{\d-\rho}
e_{\d-\rho}$, $x^{\d-\rho}\neq0$. Во всех случаях в лемме~4
полагаем $\a=\d-\rho-\s$ и выбираем подходящий коэффициент $a$. Из
этого следует, что при $\Phi = \A_3$ любой $x\in V_1$ приводится к
одному из следующих типов:

\begin{enumerate}
\item[I.] $x=0$.

\item[II.]

\begin{enumerate}

\item[а)] $x=e_\rho$;

\item[б)] $x=e_\s$.
\end{enumerate}

\item[III.] $x=e_\rho + x^\s e_\s$, $x^\s\neq0$.

\item[VI.] $x=e_\rho + x^{\d-\rho} e_{\d-\rho}$,
$x^{\d-\rho}\neq0$.
\end{enumerate}

\item Пусть $\Phi = \A_l$, $l>3$. Как и в прошлый раз, у нас
остались I, II, III и VI типы:

\begin{enumerate}
\item[I.] $x=0$.

\item[II.] $x=x^\l e_\l$; $x^\l\neq 0$.

\item[III.] $x=x^\l e_\l + x^\mu e_\mu$; $x^\l, x^\mu\neq 0$.

\item[VI.] $x=x^\l e_\l + x^{\d-\l} e_{\d-\l}$; $x^\l,
x^{\d-\l}\neq 0$.
\end{enumerate}

В случае I, разумеется, остается 0. В случае II, как и в
предыдущем пункте, все элементы приводятся к одному из двух
подтипов $x=e_\rho$ и $x=e_\s$. В случае III с помощью леммы
удается избавиться от обоих коэффициентов: после перехода к
$x=e_\rho + x^\s e_\s$, $x^\s\neq0$, аналогичному предыдущему
пункту, положим $\a = \dynkinal0010$ и выберем подходящий
коэффициент $a$. Тогда в этом случае получается один тип $x=e_\rho
+ e_\s$. В случае VI, как и в предыдущем пункте, получается
избавиться лишь от одного коэффициента, и получается серия типов
$x=e_\rho + x^{\d-\rho} e_{\d-\rho}$, $x^{\d-\rho}\neq0$. Из этого
следует, что при $\Phi = \A_l$, $l>3$, любой $x\in V_1$ приводится
к одному из следующих типов:

\begin{enumerate}
\item[I.] $x=0$.

\item[II.]

\begin{enumerate}

\item[а)] $x=e_\rho$;

\item[б)] $x=e_\s$.
\end{enumerate}

\item[III.] $x=e_\rho + e_\s$.

\item[VI.] $x=e_\rho + x^{\d-\rho} e_{\d-\rho}$,
$x^{\d-\rho}\neq0$.
\end{enumerate}
\end{itemize}

9. Докажем, что все полученные типы лежат в различных орбитах.
Случаи $\Phi=\A_1$ и $\Phi=\A_2$ очевидны. Рассмотрим случай
$\Phi=\A_3$. Тип I с нулевым вектором, как и для всех других
систем, образует одну орбиту. Для разделения остальных типов
воспользуемся рассуждениями из пункта~7. В II а) пара из
ковектора и вектора~--- это $((1,0),(0,0)^T)$; в II б)~---
$((0,0),(0,1)^T)$; в III~--- $((1,0),(0,x^\s)^T)$; в VI~---
$((1,0),(x^{\d-\rho},0)^T)$. Как несложно видеть, нулевой вектор
(и ковектор) остаются нулевыми, значит в II а) и II б) получаются
две отдельные орбиты. Таким образом, остались III и VI типы. Если
$(1,0)(g')^{-1})=(1,0)$, то первая строка $(g')^{-1}$ тоже состоит
из $1$ и $0$. Тогда первая строка $g'$ такая же, поэтому умножение
на $g'$ оставляет первую координату вектора неизменной. Это
означает, что для типа VI для каждого $x^{\d-\rho}$ действительно
получаются своя орбита. Более того, так как $g'\in\SL(2,K)$, то
матрица $g'$ будет иметь вид

$$\left(\begin{array}{cc}
1 & 0 \\
* & 1
\end{array}\right).
$$

Поэтому если первая координата вектора равна $0$, то вторая
координата тоже остается неизменной. Это доказывает, что в III
типе также при каждом $x^\s$ будет получаться своя орбита.

Для случая $\Phi=\A_l$ при $l>3$ все рассуждения точно такие же,
только отсутствуют три последних предложения, так как здесь в типе
III получалась одна орбита.

\begin{center}
{\bf \S5. Случай $\Phi=\D_l$}
\end{center}

10. Прежде чем разбирать серию $\D_l$, докажем следующую несложную
лемму:

\begin{lemma} Пусть
$$A=\left(\begin{array}{cc}
0 & p \\
q & 0
\end{array}\right), \quad
B=\left(\begin{array}{cc}
0 & s \\
t & 0
\end{array}\right).
$$

Тогда
\begin{enumerate}
\item Предположим, что $pq=k\neq 0$. Тогда $A$ и $B$ сопряжены
матрицей из $\SL(2,K)$ тогда и только тогда, когда $st=k$ и
существуют $x,y\in K$, такие что $x^2-ky^2 = ps$.

\item Предположим, что $q=0=t$. Тогда $A$ и $B$ сопряжены матрицей
из $\SL(2,K)$ тогда и только тогда, когда существует $x\in K^*$,
такие что $x^2p = s$.
\end{enumerate}
\end{lemma}
\begin{proof}
$A$ и $B$ сопряжены матрицей из $\SL(2,K)\Leftrightarrow$
существует $G\in\SL(2,K)$, такая что $GA=BG$. Расписывая это
равенство, получаем

$$\left(\begin{array}{cc}
g_{11} & g_{12} \\
g_{21} & g_{22}
\end{array}\right)\cdot\left(\begin{array}{cc}
0 & p \\
q & 0
\end{array}\right) = \left(\begin{array}{cc}
0 & s \\
t & 0
\end{array}\right)\cdot\left(\begin{array}{cc}
g_{11} & g_{12} \\
g_{21} & g_{22}
\end{array}\right).
$$
Перемножая и приравнивая обе части, получаем систему уравнений
$$\left\{\begin{array}{rcl}
g_{12}q & = & g_{21}s \\
g_{22}q & = & g_{11}t \\
g_{11}p & = & g_{22}s\\
g_{21}p & = & g_{12}t\end{array}\right. .
$$
Докажем пункт 1. Пусть $A$ и $B$ сопряжены. Тогда равенство
$st=k=pq$ очевидно из равенства определителей и, в частности, $p,q,s,t\neq 0$.
Так как $G\in\SL(2,K)$, то $g_{11}g_{22}-g_{12}g_{21} = 1$. Из
первого и третьего уравнений получаем $g_{21}=g_{12}q/s$ и
$g_{22} = g_{11}p/s$ (второе и четвертое уравнение им равносильны в силу того, что $st=pq$). Подставляя эти выражения в предыдущее
равенство, получаем $g_{11}^2p/s - g_{12}^2q/s = 1$. Умножая
полученное на $ps$, получаем нужное: $(g_{11}p)^2 - g_{12}^2pq =
ps$. Доказательство в другую сторону уже практически очевидно:
полагая $g_{12}=y$, $g_{11}=x/p$, $g_{21}=g_{12}q/s$ и $g_{22} =
g_{11}p/s$, получаем, как несложно видеть, требуемое.

Докажем пункт 2. Понятно, что достаточно разобрать случай, когда
$p$ и $s$ не равны $0$. Пусть $A$ и $B$ сопряжены. Тогда из
первого уравнения системы получаем, что $g_{21}=0$. При этом из
всех уравнений системы остается лишь третье: $g_{11}p=g_{22}s$.
Так как $G\in\SL(2,K)$, то $g_{11}g_{22} - g_{12}g_{21} = 1$,
откуда $g_{22}=1/g_{11}$. Подставляя в третье уравнение, получаем
$g_{11}p=s/g_{11}$, откуда следует нужное равенство.
Доказательство в другую сторону тоже очевидно: полагая $g_{11}=x$,
$g_{22}=1/g_{11}$, $g_{21}=0$, а $g_{12}$ любым, получим
требуемое.
\end{proof}

11. Объясним, где будет использоваться эта лемма. Пусть $\Phi=\D_4$
и $x\in V_1$. По лемме~5, этому элементу соответствует корневой
элемент $y = \sum_{\a\in\Phi} y^\a e_\a + y^h$, причем $y^\a=x^\a$
при $\a\in\Phi_1$, $y^\d = 1$ и $y^h\in\langle h_\a;
\a\in\Phi_0\cap\Pi\rangle$. Как следует из утверждения~2 [17],
$$y = x_{-\d}(\cdot) x_{\l-\d}(N_{\l-\d,\d}x^\l)
x_{\rho-\d}(N_{\rho-\d,\d}x^\rho) x_{\s-\d}(N_{\s-\d,\d}x^\s)
x_{\tau-\d}(N_{\tau-\d,\d}x^\tau) e_\d.$$ При этом, как несложно
видеть, коэффициент при $x_{-\d}$ должен быть равен $0$ (чтобы $y^h\in\langle h_\a;
\a\in\Phi_0\cap\Pi\rangle$), и $y^h$
будет равно $0$. Тогда $y^{\l+\rho-\d} =
N_{\l-\d,\d}N_{\l-\d,\rho} x^\l x^\rho$, $y^{\d-\l-\rho} =
N_{\s-\d,\d} N_{\s-\d,\tau} x^\s x^\tau$, $y^{\l+\s-\d} =
N_{\l-\d,\d} N_{\l-\d,\s} x^\l x^\s$, $y^{\d-\l-\s} =
N_{\rho-\d,\d} N_{\rho-\d,\tau} x^\rho x^\tau$, $y^{\l+\tau-\d} =
N_{\l-\d,\d} N_{\l-\d,\tau} x^\l x^\tau$, $y^{\d-\l-\tau} =
N_{\rho-\d,\d} N_{\rho-\d,\s} x^\rho x^\s$. Как уже отмечалось,
при умножении на $g\in G_0$ элементу $gx$ соответствует элемент
$gy$, поэтому можно вместо элемента $x$ изучать $y$.

Положим $V_0=\langle e_\a, \a\in\Phi_0; h_\a,
\a\in\Phi_0\cap\Pi\rangle$. Далее, пусть $z\in V_0$, $z^\a = y^\a$
при $\a\in\Phi_0$ и $z^h=y^h=0$. Как несложно видеть, при
умножении на $g\in G_0$ элементу $gy$ соответствует элемент $gz$.
Так как в нашем случае $\Phi_0 = \A_1\times \A_1\times \A_1$, то $z$
можно представить в виде $z_1+z_2+z_3$, где каждый $z_i$ из своей
подалгебры типа $\A_1$; при этом $z_1^h + z_2^h + z_3^h = z^h = 0$,
поэтому все $z_i^h = 0$. Наконец, $g\in G_0$ можно разложить
схожим образом на произведение $g_1g_2g_3$, где каждый $g_i$
принадлежит своей группе $G_{\text{\rm sc}}(\A_1,K)\cong\SL(2,K)$;
при этом каждый элемент $g_i$ действует только на "своем"\ $z_i$:
$g_1g_2g_3(z_1+z_2+z_3) = g_1z_1 + g_2z_2 + g_3z_3$. Отметим, что
$g_i$ друг с другом никак не связаны и могут подбираться
независимо. В естественном представлении $z_i$ представляются
матрицами $2\times 2$, а $g_i\in\SL(2,K)$ действуют на них
сопряжением. Для подсчитанного ранее элемента $y$ получаем:

$$z_1 = \left(\begin{array}{cc}
 0              & N_{\l-\d,\d}N_{\l-\d,\rho} x^\l x^\rho \\
N_{\s-\d,\d} N_{\s-\d,\tau} x^\s x^\tau & 0
\end{array}\right),$$
$$z_2 = \left(\begin{array}{cc}
 0                & N_{\l-\d,\d} N_{\l-\d,\s} x^\l x^\s \\
N_{\rho-\d,\d} N_{\rho-\d,\tau} x^\rho x^\tau & 0
\end{array}\right),$$
$$z_3 = \left(\begin{array}{cc}
 0              & N_{\l-\d,\d} N_{\l-\d,\tau} x^\l x^\tau \\
N_{\rho-\d,\d} N_{\rho-\d,\s} x^\rho x^\s & 0
\end{array}\right).
$$

Здесь можно воспользоваться леммой~6. Подробно для каждого
случая мы рассмотрим ее в пункте 16; отметим лишь, что
коэффициент $k$, возникающий в лемме, для $z_1$ равен
$(N_{\l-\d,\d}N_{\l-\d,\rho} x^\l x^\rho)\cdot(N_{\s-\d,\d}
N_{\s-\d,\tau} x^\s x^\tau)$, что равно $x^\l x^\rho x^\s x^\tau$
по лемме~2. Для $z_2$, аналогичными рассуждениями, получается то же самое выражение. Нужно также обратить внимание, что если переход от $x$ к $y$ обратим, то от $y$ к $z$~--- уже нет: в один $z$ могут переходить разные $y$ (и, соответственно, разные $x$). 

Если $\Phi=\D_l$, $l>4$, то рассуждение похожее. Так как сейчас
$\Phi_0=\A_1\times \D_{l-2}$ (для простоты будем полагать $\D_3=\A_3$,
чтобы не рассматривать случай $l=5$ отдельно), то $z$ можно
представить в виде $z_1+z_2$, где $z_1$ из подалгебры типа $\A_1$,
а $z_2$~--- типа $\D_{l-2}$; при этом $z^h$ также представится в
виде $z_1^h + z_2^h$, и все они будут равны $0$. Продолжая рассуждения, мы обнаруживаем, что
матрица $z_1$, как и в прошлый раз, попадает в условие леммы. С
$z_2$ ситуация получается другой, больше схожей с рассмотренными в
[18] случаями $\Phi=\E_6$, $\E_7$ и $\E_8$.

Использоваться это будет, в первую очередь, для доказательства того, что два элемента
не лежат в одной орбите. В самом деле, если $x, x'\in V_1$
переходят друг в друга, то и соответствующие им корневые элементы
$y, y'\in V$ тоже, а значит и соответствующие им элементы $z,
z'\in V_0$ тоже. При этом $z$ и $z'$ распадаются на более простые
части, и эти части также будут переходить друг в друга.
Соответственно, если, скажем, $z_1$ и $z_1'$ друг в друга не
переходят, то $x$ и $x'$ лежат в разных орбитах.

12. Продолжим изучение вопроса существования и произвольности
выбора четырех попарно ортогональных корней $\l$, $\mu$, $\nu$ и
$\xi$ для разных систем корней, сейчас для $\Phi=\D_l$. Как и
прежде, мы перечислим орбиты этих корней под действием группы
Вейля системы $\Phi_0$; так как корни выбирались в теореме по
порядку, то порядок нам тоже важен. Таким образом, для определения
корня $\l$ мы можем пользоваться всей группой Вейля, для корня
$\mu$~--- только подгруппой, оставляющей уже выбранный $\l$
неподвижным, для корня $\nu$~--- подгруппой, оставляющей $\l$ и
$\mu$ неподвижными. Корень $\xi$ определяется по корням $\l$,
$\mu$ и $\nu$ однозначно, поэтому про него можно и не говорить (к
слову, подгруппа, оставляющая на месте $\l$, $\mu$ и $\nu$, будет
единичной).

\begin{itemize}
\item Пусть $\Phi = \D_4$. В этом случае $\Phi_1$ состоит из 8
корней $\dynkindfour0010$, $\dynkindfour0011$, $\dynkindfour0110$,
$\dynkindfour0111$, $\dynkindfour1010$, $\dynkindfour1011$,
$\dynkindfour1110$ и $\dynkindfour1111$; при этом у корня $\l$
получается одна орбита; можно считать, что $\l=\dynkindfour0010$.
С корнями $\mu$, $\nu$ и $\xi$ ситуация интереснее~--- множество
$\{\mu, \nu, \xi\}$ однозначно определяется выбором $\l$, но друг
в друга корни $\mu$, $\nu$ и $\xi$ не переводятся; поэтому у корня
$\mu$ получается три орбиты, у $\nu$ две орбиты, а у $\xi$~---
одна орбита. Обозначим корень $\dynkindfour0111$ через $\rho$,
$\dynkindfour1011$~--- через $\s$, а $\dynkindfour1110$~--- через
$\tau$. Тогда $\{\mu,\nu,\xi\} = \{\rho,\s,\tau\}$, и возможны все
6 перестановочных вариантов.

\item Пусть $\Phi = \D_l$, $l>4$. В этом случае $\Phi_1$ состоит из
$4l-8$ корней вида $\dynkindeight***\dots **1*$. Несложно видеть,
что все эти корни переводятся друг в друга, то есть у корня $\l$
одна орбита; можно считать, что $\l = \dynkindeight000\dots 0010$.
Корни из $\Phi_1$, ортогональные $\l$~--- это либо
$\dynkindeight112\dots2210$, либо корни вида
$\dynkindeight***\dots*111$; при этом первый корень ортогонален
всем остальным. Корни вида $\dynkindeight***\dots*111$ переводятся
друг в друга, поэтому из них можно выбирать любой; к корню
$\dynkindeight000\dots0111$, скажем, в этом множестве есть ровно
один ортогональный $\dynkindeight112\dots2111$. Итого, получается
что у корня $\mu$ есть две орбиты, причем в одной из них у $\nu$ и
$\xi$ по одной орбите, а в другой у $\nu$ две орбиты, а у $\xi$
одна; у множества $\{\l, \mu, \nu, \xi\}$ одна орбита. Обозначим
корень $\dynkindeight112\dots2210$ через $\rho$,
$\dynkindeight000\dots0111$~--- через $\s$, а
$\dynkindeight112\dots2111$~ через $\tau$. Тогда можно считать,
что либо $\mu=\rho$, $\nu=\s$ и $\xi=\tau$, либо $\mu=\s$,
$\nu=\rho$ и $\xi=\tau$, либо $\mu=\s$, $\nu=\tau$ и $\xi=\rho$.
\end{itemize}

13. На самом деле, однако, нас интересуют немного другие орбиты, а
именно, орбиты множеств $\{\l\}$, $\{\l, \mu\}$, $\{\l, \mu,
\nu\}$ и $\{\l, \mu, \nu, \xi\}$. Иначе говоря, порядок корней нам
не важен. Это, разумеется, более слабое условие и некоторые орбиты
могут объединиться. Для рассмотренных ранее случаев $\Phi=\E_l$ (в
[18]) и $\Phi=\A_l$ этого не происходило, но в рассматриваемом случае
это кое-где происходит.

\begin{itemize}
\item Пусть $\Phi=\D_4$. Тогда из изложенного в предыдущем пункте
получается:

\begin{enumerate}
\item[1)] у корня $\l$ одна орбита;

\item[2)] пару корней $\{\l, \mu\}$ можно перевести в $\{\l,
\rho\}$, $\{\l, \s\}$ или $\{\l, \tau\}$;

\item[3)] тройку корней $\{\l, \mu, \nu\}$ можно перевести в
$\{\l, \rho, \s\}$, $\{\l, \rho, \tau\}$ или $\{\l, \s, \tau\}$;

\item[4)] у четверки корней $\{\l, \mu, \nu, \xi\}$ одна орбита.
\end{enumerate}

Оказывается, в случае 3) у нас всего одна орбита, и в этом
несложно убедиться непосредственно. А именно, тройка $\{\l, \rho,
\s\}$ под действием элемента группы Вейля, соответствующего корню
$\l+\rho-\d$, переходит в $\{\d-\rho, \d-\l, \d-\tau\}$;
полученная тройка под действием элемента группы Вейля,
соответствующего корню $\l+\tau-\d$, переходит в $\{\s, \tau,
\l\}$. Для оставшейся тройки рассуждение аналогичное.

\item Пусть $\Phi=\D_l$, $l>4$. Тогда из изложенного в предыдущем
пункте получается:

\begin{enumerate}
\item[1)] у корня $\l$ одна орбита;

\item[2)] пару корней $\{\l, \mu\}$ можно перевести в $\{\l,
\rho\}$ или $\{\l, \s\}$;

\item[3)] тройку корней $\{\l, \mu, \nu\}$ можно перевести в
$\{\l, \rho, \s\}$ или $\{\l, \s, \tau\}$;

\item[4)] у четверки корней $\{\l, \mu, \nu, \xi\}$ одна орбита.
\end{enumerate}

Повторяя дословно рассуждение из случая $\Phi=\D_4$, мы снова
обнаруживаем, что в 3), на самом деле, одна орбита.

\end{itemize}

14. Пусть $\Phi = \D_4$. Тогда $\Phi_1$ состоит из восьми векторов
$\l$, $\rho$, $\s$, $\tau$, $\d-\l$, $\d-\rho$, $\d-\s$ и
$\d-\tau$, образующих куб (к слову, он уже обсуждался в статье
[18]); $\Phi_0 = \{\pm(\d-\l-\rho), \pm(\d-\l-\s),
\pm(\d-\l-\tau)\}=\A_1\times \A_1\times \A_1$. В этом случае у нас
остались первые пять типов:

\begin{enumerate}
\item[I.] $x=0$.

\item[II.] $x=x^\l e_\l$; $x^\l\neq 0$.

\item[III.] $x=x^\l e_\l + x^\mu e_\mu$; $x^\l, x^\mu\neq 0$.

\item[IV.] $x=x^\l e_\l + x^\mu e_\mu + x^\nu e_\nu$; $x^\l,
x^\mu, x^\nu\neq 0$.

\item[V.] $x=x^\l e_\l + x^\mu e_\mu + x^\nu e_\nu +x^\xi e_\xi$;
$x^\l, x^\mu, x^\nu, x^\xi\neq 0$.
\end{enumerate}

Как мы доказывали в [18] (и уже упоминали в пункте 5), в случае II
вектор $x$ получается сингулярным, в случае III~--- блестящим, в
случае IV~--- светящимся, а в случае V~--- темным. Кроме того, в
случае V произведение всех коэффициентов $x^\l x^\mu x^\nu x^\xi$
является неизменным. Это означает, что элементы $x$ из разных
типов точно лежат в различных орбитах. Осталось понять, как
обстоят дела внутри каждого типа.

По пункту 13 мы получаем, что при переходе к конкретным корням
тип III распадается на три подтипа.

\begin{enumerate}
\item[I.] $x=0$.

\item[II.] $x=x^\l e_\l$; $x^\l\neq 0$.

\item[III.]

\begin{enumerate}

\item[a)] $x=x^\l e_\l + x^\rho e_\rho$; $x^\l, x^\rho\neq 0$;

\item[б)] $x=x^\l e_\l + x^\s e_\s$; $x^\l, x^\s\neq 0$;

\item[в)] $x=x^\l e_\l + x^\tau e_\tau$; $x^\l, x^\tau\neq 0$.
\end{enumerate}

\item[IV.] $x=x^\l e_\l + x^\rho e_\rho + x^\s e_\s$; $x^\l,
x^\rho, x^\s\neq 0$.

\item[V.] $x=x^\l e_\l + x^\rho e_\rho + x^\s e_\s +x^\tau
e_\tau$; $x^\l, x^\rho, x^\s, x^\tau\neq 0$.
\end{enumerate}

Отметим, однако, что пока мы еще не доказали, что элементы из
разных подтипов лежат в разных орбитах. Это будет сделано чуть
позднее.

15. Как несложно видеть, если взять вектор $x=x^\l e_\l + x^\rho
e_\rho + x^\s e_\s +x^\tau e_\tau$ (будем считать, что $x^\l\neq
0$, а остальные коэффициенты произвольные), то при применении к
нему любого $w_\a(a)$ из леммы~4, $\a\in\Phi_0$, два его
коэффициента умножатся на $a$, а два других~--- на $\frac1a$.
Тогда, применяя это действие три раза для трех ортогональных
корней из $\Phi_0$, из вектора $x$ можно получить $x^\l {abc} e_\l
+ \frac {x^\rho a}{bc} e_\rho + \frac{x^\s b}{ac} e_\s
+\frac{x^\tau c}{ab} e_\tau$. Если положить $c=\frac1{x^\l ab}$,
чтобы коэффициент при $e_\l$ стал равен $1$, то $x^\l {abc} e_\l +
\frac {x^\rho a}{bc} e_\rho + \frac{x^\s b}{ac} e_\s +\frac{x^\tau
c}{ab} e_\tau = e_\l + x^\l x^\rho a^2 e_\rho + x^\l x^\s b^2 e_\s
+ \frac{x^\tau}{x^\l a^2 b^2} e_\tau$. Таким образом, $x^\l$ можно
сделать равным $1$, а два других можно умножать на произвольные
ненулевые квадраты.

16. Рассмотрим классификацию из 14 пункта. С типом I все
понятно. В типе II по пункту 15 всего одна орбита, поэтому с
ним тоже нет никаких проблем. В типе III из матриц $z_1$, $z_2$ и
$z_3$ две матрицы нулевые, а одна имеет ненулевой коэффициент над
диагональю, то есть подпадает под условие леммы 6, 2): в III а)
это $z_1$, в III б)~--- $z_2$, а в III в)~--- $z_3$. Так как
нулевая матрица при сопряжении переходит в нулевую, а ненулевая
нет, то все три подтипа действительно разные. По лемме 6, 2) и виду $z_i$, если два вектора из одного подтипа лежат в одной орбите, то произведение коэффициентов у них лежат
в одном классе $K^*/(K^*)^2$. А из пункта~15 следует обратное: если произведение коэффициентов у них лежат
в одном классе $K^*/(K^*)^2$, то векторы лежат в одной орбите. Таким образом, эти два утверждения равносильны. Если при выборе представителя каждой
орбиты сделать $x^\l$ равным $1$, то оставшийся коэффициент можно
выбирать из $K^*/(K^*)^2$.

Аналогично, в типе IV матрицы $z_1$, $z_2$ и $z_3^T$ имеют
ненулевой коэффициент над диагональю, то есть подпадают под
условие леммы~6, 2) (отметим, что если матрицы сопряжены, то и
транспонированные к ним матрицы также сопряжены). Соответственно,
по лемме~6, 2), виду $z_i$ и пункту~15, если при выборе представителя
каждой орбиты сделать $x^\l$ равным $1$, то два оставшихся
коэффициента можно выбирать из $K^*/(K^*)^2$.

Осталось рассмотреть случай, когда вектор $x$ имеет тип V. До сих
пор мы использовали переход к матрицам $z_i$ и лемму~6 для
доказательства лишь в одну сторону~--- что если элементы $x$ и
$\tilde x$ сопряжены, то соответствующие $z_i$ также сопряжены, и
выполняется условие на коэффициенты из леммы~6. Для
доказательства в другую сторону нам хватало пункта~15. В данном
случае ситуация иная~--- лемма~6 и пункт~15 дают различные
оценки. Поэтому вместо рассуждений из пункта~15 мы докажем, что
если соответствующие $z_i$ сопряжены, то и изначальные элементы
также сопряжены.

В типе V все матрицы $z_i$ имеют ненулевые коэффициенты как над,
так и под диагональю, то есть подпадают под условие леммы 6,
1). Предположим, что мы сумели матрицы $z_1$ и $z_2$ перевести
сопряжением, как в лемме 6, 1), в матрицы
$$\tilde z_1 = \left(\begin{array}{cc}
 0    & \pm a \\
\pm b & 0
\end{array}\right),
\tilde z_2 = \left(\begin{array}{cc}
 0    & \pm c \\
\pm d & 0
\end{array}\right),
$$
где знаки $\pm$ в каждой позиции совпадает с соответствующим
знаком в матрицах $z_1$ и $z_2$ в пункте 11; $a$, $b$, $c$ и
$d$ при этом будут отличны от $0$. Докажем, что тогда вектор $x$
можно перевести в вектор $e_\l + a e_\rho + c e_\s + \frac bc
e_\tau$.

Когда мы перевели, как в лемме 6, 1), матрицы $z_1$ и $z_2$ в
$\tilde z_1$ и $\tilde z_2$, вектор $x$ перешел в некоторый вектор
$\tilde x = \tilde x^\l e_\l + \tilde x^\rho e_\rho + \tilde x^\s
e_\s + \tilde x^\tau e_\tau + \tilde x^{\d-\l} e_{\d-\l} + \tilde
x^{\d-\rho} e_{\d-\rho} + \tilde x^{\d-\s} e_{\d-\s} + \tilde
x^{\d-\tau} e_{\d-\tau}$. Как и раньше, обозначим через $\tilde y$
соответствующий ему по лемме 5 элемент из $V$. Напомним, что соответствующий, в свою очередь, этому $\tilde y$ элемент $\tilde z$ распадается на те самые $\tilde z_1$, $\tilde z_2$ и $\tilde z_3$, поэтому $a=\pm\tilde y^{\l+\rho-\d}$, $b=\pm\tilde y^{\d-\l-\rho}$, $c=\pm\tilde y^{\l+\s-\d}$ и $d=\pm\tilde y^{\d-\l-\s}$. 
Заметим, что действие корневыми элементами $x_{\l+\tau-\d}(\cdot)$ и
$x_{\d-\l-\tau}(\cdot)$ не меняют коэффициенты в разложении $\tilde y$ при элементах
$e_{\l+\rho-\d}, e_{\d-\l-\rho}, e_{\l+\s-\d}$ и $e_{\d-\l-\s}$ и, соответственно, коэффициенты $a, b, c$ и $d$, как и все матрицы $\tilde z_1$ и $\tilde z_2$ в целом. Можно также отметить, что $\Phi_0$ состоит из трех пар корней, и соответствующие каждой паре корневые элементы меняют только одну $\tilde z_i$ из трех; в частности, $x_{\l+\tau-\d}(\cdot)$ и $x_{\d-\l-\tau}(\cdot)$ меняют $\tilde z_3$.

Как несложно видеть, $0 \neq a = \pm\tilde y^{\l+\rho-\d} = \pm \tilde x^\l \tilde
x^\rho \pm \tilde x^{\d-\s} \tilde x^{\d-\tau}$, поэтому $\tilde x^\l$ или $\tilde x^{\d-\tau}$ отличен от $0$. Действуя, при необходимости, $x_{\d-\l-\tau}(\cdot)$, можно считать, что $\tilde x^{\d-\tau}\neq 0$. Далее, действуя $x_{\l+\tau-\d}(\cdot)$, можно добиться того, чтобы $\tilde x^\l=1$ и, снова действуя $x_{\d-\l-\tau}(\cdot)$, чтобы $\tilde x^{\d-\tau} = 0$. В силу вышеуказанной формулы $a = \pm \tilde x^\l \tilde x^\rho \pm \tilde x^{\d-\s}
\tilde x^{\d-\tau}$ и выбора знака при определении $a$, получаем,
что $\tilde x^\rho = a$. Тогда, действуя корневым элементом
$x_{\l+\tau-\d}(\cdot)$, можно добиться того, что $\tilde
x^{\d-\s} = 0$. Теперь $\tilde x = e_\l + a e_\rho + \tilde x^\s
e_\s + \tilde x^\tau e_\tau + \tilde x^{\d-\l} e_{\d-\l} + \tilde
x^{\d-\rho} e_{\d-\rho}$. Аналогично, так как $c = \pm \tilde x^\l
\tilde x^\s \pm \tilde x^{\d-\rho}\tilde x^{\d-\tau}$, мы
получаем, что $\tilde x^\s = c$. Наконец, действуя корневым
элементом $x_{\d-\l-\rho}(\cdot)$, можно также добиться того, что
$\tilde x^{\d-\rho} = 0$; выбранные ранее коэффициенты $\tilde x^\l$, $\tilde x^\rho$, $\tilde x^\s$ (так как $\tilde x^{\d-\tau} = 0$), $\tilde x^{\d-\tau}$ и $\tilde x^{\d-\s}$ при этом не меняются. Матрица $\tilde z_1$ изменится, но она
нам уже не нужна.

Таким образом, теперь $\tilde x = e_\l + a e_\rho + c e_\s +
\tilde x^\tau e_\tau + \tilde x^{\d-\l} e_{\d-\l}$. Для того чтобы избавиться от последнего слагаемого, сошлемся на вычисления из [18], пункт~7, упрощение IX случая. Они проходят дословно, только в той статье элемент записывался в виде $x=x^\l e_\l + x^{\d-\l} e_{\d-\l} + x^\mu e_\mu + x^\nu e_\nu + x^\xi e_\xi$, поэтому нужно сделать замены $\mu:=\rho$, $\nu:=\tau$, $\xi:=\s$ и, соответственно, $x^\l:=1$, $x^{\d-\l}:=\tilde x^{\d-\l}$, $x^\mu:=a$, $x^\nu:=\tilde x^\tau$ и $x^\xi:=c$. В [18] требовалось, чтобы все коэффициенты были не равны $0$, но, на самом деле, в тех рассуждениях $x^\tau$ вполне может равняться $0$, так как он нигде не оказывается в знаменателе. После выполнения вычислений из [18] получался вектор $\frac{x^\l}{x^\mu} e_\l + e_\mu + (x^\nu x^\mu + \frac{x^\l
(x^{\d-\l})^2}{4x^\xi}) e_\nu + x^\xi x^\mu e_\xi$; если же на
полученный вектор подействовать элементом $w_{\d-\l-\mu}(x^\mu)$
(из леммы 4), то получится вектор $x^\l e_\l + x^\mu e_\mu +
(x^\nu + \frac{x^\l (x^{\d-\l})^2}{4x^\xi x^\mu}) e_\nu + x^\xi
e_\xi$. Возвращаясь к нашим изначальным обозначениям, получаем вектор $e_\l + a e_\rho +
(\tilde x^\tau + \frac{(\tilde x^{\d-\l})^2}{4ac}) e_\tau + c e_\s$. Осталось заметить, что
произведение коэффициентов, как мы отмечали в пункте~5,
остается неизменным. Так как, по выбору $a$, $b$, $c$ и $d$, $x^\l
x^\rho x^\s x^\tau = ab = cd$, то коэффициент при $e_\tau$ равен $\frac bc$, что
и требовалось.

Объединяя полученное с леммой 6, 1) и пунктом 11, получаем,
что два элемента $x = e_\l + x^\rho e_\rho + x^\s e_\s + x^\tau
e_\tau$ и $\tilde x = e_\l + \tilde x^\rho e_\rho + \tilde x^\s
e_\s + \tilde x^\tau e_\tau$ лежат в одной орбите тогда и только
тогда, когда $k := x^\rho x^\s x^\tau = \tilde x^\rho \tilde x^\s
\tilde x^\tau$ и существуют $x, y, x', y'\in K$, такие что
$x^2-ky^2=x^\rho \tilde x^\rho$ и $x'^2-ky'^2=x^\s \tilde x^\s$.

17. Полученные результаты можно объединить в теорему: 

\begin{theorem}

Пусть $\Phi=\D_4$ и $\Char K\neq 2$. Обозначим, для определенности,
$\l = \dynkindfour0010$, $\rho = \dynkindfour0111$, $\s =
\dynkindfour1011$ и $\tau = \dynkindfour1110$. Далее, пусть
$a\sim_k b \Leftrightarrow \exists x,y \in K : x^2-ky^2=ab$~---
отношение эквивалентности на $K^*$ и $K_k = K^*/\sim_k$~--
множество классов по этому отношению. Тогда векторы из $V_1$ под
действием $G_0$ образуют следующие орбиты:

\begin{enumerate}
\item[I.] Нулевая орбита, $x=0$.

\item[II.] Одна орбита из сингулярных векторов. Ее элементы можно
привести к виду $x=e_\l$.

\item[III.] Три серии орбит из блестящих векторов. Их элементы
можно привести к одному из следующих видов:
\begin{enumerate}
\item[a)] $x = e_\l + x^\rho e_\rho$; $x^\rho\in K^*/(K^*)^2$;

\item[б)] $x = e_\l + x^\s e_\s$; $x^\s\in K^*/(K^*)^2$;

\item[в)] $x = e_\l + x^\tau e_\tau$; $x^\tau\in K^*/(K^*)^2$.
\end{enumerate}

\item[IV.] Одна серия орбит из светящихся векторов. Их элементы
можно привести к виду $x = e_\l + x^\rho e_\rho + x^\s e_\s$;
$x^\rho, x^\s\in K^*/(K^*)^2$.

\item[V.] Одна серия орбит из темных векторов. Их элементы можно
привести к виду $x=e_\l + x^\rho e_\rho + x^\s e_\s + x^\tau
e_\tau$, где $k\in K^*$, $x^\rho, x^\s\in K_k$, $x^\tau = \frac
k{x^\rho x^\s}$.
\end{enumerate}

\end{theorem}

18. Пусть $\Phi=\D_l$, $l>4$. Этот случай, разумеется, похож на
$\Phi=\D_4$, но есть некоторые отличия. Здесь $\Phi_1$ состоит из
$4l-8$ векторов вида $\dynkindeight***\dots **1*$, $\Phi_0 =
\A_1\times \A_{l-2}$. В обозначениях из пункта~12 $\l =
\dynkindeight000\dots 0010$, $\rho = \dynkindeight112\dots2210$,
$\s = \dynkindeight000\dots0111$ и $\tau =
\dynkindeight112\dots2111$, $\A_1$ из разложения $\Phi_0$ совпадает с
$\{\pm(\d-\l-\rho)\}$. У нас снова остались первые пять типов:

\begin{enumerate}
\item[I.] $x=0$.

\item[II.] $x=x^\l e_\l$; $x^\l\neq 0$.

\item[III.] $x=x^\l e_\l + x^\mu e_\mu$; $x^\l, x^\mu\neq 0$.

\item[IV.] $x=x^\l e_\l + x^\mu e_\mu + x^\nu e_\nu$; $x^\l,
x^\mu, x^\nu\neq 0$.

\item[V.] $x=x^\l e_\l + x^\mu e_\mu + x^\nu e_\nu +x^\xi e_\xi$;
$x^\l, x^\mu, x^\nu, x^\xi\neq 0$.
\end{enumerate}

Как уже говорилось, в случае II вектор $x$ получается сингулярным,
в случае III~--- блестящим, в случае IV~--- светящимся, а в случае
V~--- темным. Кроме того, в случае V произведение всех
коэффициентов $x^\l x^\mu x^\nu x^\xi$ является неизменным. Это
означает, что элементы $x$ из разных типов точно лежат в различных
орбитах. Осталось понять, как обстоят дела внутри каждого типа.

По пункту 13 мы получаем, что при переходе к конкретным корням
тип III распадается на два подтипа.

\begin{enumerate}
\item[I.] $x=0$.

\item[II.] $x=x^\l e_\l$; $x^\l\neq 0$.

\item[III.]

\begin{enumerate}

\item[a)] $x=x^\l e_\l + x^\rho e_\rho$; $x^\l, x^\rho\neq 0$;

\item[б)] $x=x^\l e_\l + x^\s e_\s$; $x^\l, x^\s\neq 0$.
\end{enumerate}

\item[IV.] $x=x^\l e_\l + x^\rho e_\rho + x^\s e_\s$; $x^\l,
x^\rho, x^\s\neq 0$.

\item[V.] $x=x^\l e_\l + x^\rho e_\rho + x^\s e_\s +x^\tau
e_\tau$; $x^\l, x^\rho, x^\s, x^\tau\neq 0$.
\end{enumerate}

19. Пункт 15 выполняется и для случая $\Phi=\D_l$, $l>4$, то
есть из элемента $x=x^\l e_\l + x^\rho e_\rho + x^\s e_\s +x^\tau
e_\tau$ можно получить $x^\l {abc} e_\l + \frac {x^\rho a}{bc}
e_\rho + \frac{x^\s b}{ac} e_\s +\frac{x^\tau c}{ab} e_\tau$.
Однако в рассматриваемом случае появляются и другие возможные изменения. А именно, как
несложно видеть, корень $\a=\a_{l-3}$ ортогонален корням $\l$ и
$\rho$, а также образует угол $2\pi/3$ с $\s$ и $\pi/3$ с $\tau$;
поэтому при применении к $x$ элемента $w_{\a}(a)$ из леммы~4
получится $x^\l e_\l + x^\rho e_\rho + x^\s a e_\s +\frac
{x^\tau}a e_\tau$. Аналогично, корень $\b=\dynkindeight111\dots
1100$ будет ортогонален корням $\s$ и $\tau$, а также образовывать
угол $2\pi/3$ с $\l$ и $\pi/3$ с $\rho$; поэтому при применении к
$x$ элемента $w_{\b}(a)$ из леммы~4 получится $x^\l a e_\l +
\frac{x^\rho}a e_\rho + x^\s e_\s +x^\tau e_\tau$.

20 (18). Так как $\Phi_0 = \A_1\times \A_{l-2}$, то, как и для случая
$\D_4$, элементу $x$ можно сопоставить матрицу $z_1$,
соответствующую подсистеме $\A_1$. Как и прежде, если элементы $x,
x'\in V_1$ переходят друг в друга, то и соответствующие им
элементы $z_1, z'_1\in V_0$ тоже.

Наконец, заметим, что корни $\l, \rho, \s$ и $\tau$ порождают
подсистему типа $\D_4$, и для нее верны проделанные выше
рассуждения. То есть если элементы $x$ и $x'$ лежали в одной
орбите в случае $\Phi=\D_4$, то в случае $\Phi=\D_l$, $l>4$, они
тоже будут лежать в одной орбите. Наоборот, разумеется, это не
верно~--- различные орбиты из $\D_4$ могут "склеиться"{} в $\D_l$,
так как действующая группа там будет больше.

21. Рассмотрим классификацию из 18 пункта. С типом I все
понятно. В типе II по пункту 19 всего одна орбита, поэтому с
ним тоже нет никаких проблем. В подтипе III а) матрица $z_1$ имеет
ненулевой коэффициент над диагональю, то есть подпадает под
условие леммы 6, 2), а в III б) матрица $z_1$ нулевая, поэтому
элементы из этих подтипов лежат в разных орбитах. В III а), по
лемме 6, 2) и пункту 19, два вектора лежат в одной орбите
тогда и только тогда, когда произведение коэффициентов у них лежат
в одном классе $K^*/(K^*)^2$. Если при выборе представителя каждой
орбиты сделать $x^\l$ равным $1$, то оставшийся коэффициент можно
выбирать из $K^*/(K^*)^2$. Что касается III б), то по пункту
19, в нем всего одна орбита.

В типе IV матрица $z_1$ опять имеет ненулевой коэффициент над
диагональю, то есть подпадает под условие леммы 6, 2).
Соответственно, по лемме 6, 2) и пункту 19, если при выборе
представителя каждой орбиты сделать $x^\l$ равным $1$, то $x^\rho$
можно выбирать из $K^*/(K^*)^2$. При этом коэффициент $x^\s$, тоже
по пункту 19, можно сделать любым.

Осталось разобраться с типом V. Пользуясь, как мы говорили в
пункте 20, результатом для $\D_4$, мы получаем, что любой темный
вектор можно привести к виду $x=e_\l + x^\rho e_\rho + x^\s e_\s +
x^\tau e_\tau$, где $k\in K^*$, $x^\rho, x^\s\in K_k$, $x^\tau =
\frac k{x^\rho x^\s}$. По пункту 19 можно $x^\s$ сделать
произвольным ненулевым числом, меняя лишь коэффициент $x^\tau$; мы положим $x^\s$ равным $1$. Таким
образом, темный вектор можно привести к виду $x=e_\l + x^\rho
e_\rho + e_\s + x^\tau e_\tau$, где $k\in K^*$, $x^\rho\in K_k$,
$x^\tau = \frac k{x^\rho}$. Докажем, что все такие элементы лежат
в разных орбитах. Прежде всего напомним, что число $k$, равное
произведению всех коэффициентов, является инвариантом. Далее,
матрица $z_1$ подпадает под условие леммы  6, 1), причем
коэффициент над диагональю равен $N_{\l-\d,\d}N_{\l-\d,\rho}
x^\rho$, а число $k$ совпадает с также обозначенным числом из
леммы. Поэтому $x^\rho\in K_k$, что и требовалось.

22. Полученные результаты можно объединить в теорему: 

\begin{theorem}

Пусть $\Phi=\D_l$, $l>4$, и $\Char K\neq 2$. Обозначим, для
определенности, $\l = \dynkindeight000\dots 0010$, $\rho =
\dynkindeight112\dots2210$, $\s = \dynkindeight000\dots0111$ и
$\tau = \dynkindeight112\dots2111$. Далее, пусть $a\sim_k b
\Leftrightarrow \exists x,y \in K : x^2-ky^2=ab$~--- отношение
эквивалентности на $K^*$ и $K_k = K^*/\sim_k$~-- множество классов
по этому отношению. Тогда векторы из $V_1$ под действием $G_0$
образуют следующие орбиты:

\begin{enumerate}
\item[I.] Нулевая орбита, $x=0$.

\item[II.] Одна орбита из сингулярных векторов. Ее элементы можно
привести к виду $x=e_\l$.

\item[III.] Одна серия орбит и одна отдельная орбита из блестящих
векторов. Их элементы можно привести к одному из следующих видов:
\begin{enumerate}
\item[a)] $x = e_\l + x^\rho e_\rho$; $x^\rho\in K^*/(K^*)^2$;

\item[б)] $x = e_\l + e_\s$.
\end{enumerate}

\item[IV.] Одна серия орбит из светящихся векторов. Их элементы
можно привести к виду $x = e_\l + x^\rho e_\rho + e_\s$;
$x^\rho\in K^*/(K^*)^2$.

\item[V.] Одна серия орбит из темных векторов. Их элементы можно
привести к виду $x=e_\l + x^\rho e_\rho + e_\s + x^\tau e_\tau$,
где $k\in K^*$, $x^\rho\in K_k$, $x^\tau = \frac k{x^\rho}$.
\end{enumerate}

\end{theorem}

\begin{center}
{\bf Литература}
\end{center}

\begin{enumerate}

\item Борель А., {\it Свойства и линейные представления групп
Шевалле}, Семинар по алгебраическим группам, Мир, М., 1973, с.
9--59.

\item Бурбаки Н., {\it Группы и алгебры Ли}, Главы IV -- VI, Мир,
М., 1972.

\item Бурбаки Н., {\it Группы и алгебры Ли}, Главы VII -- VIII,
Мир, М., 1978.

\item Вавилов Н. А., {\it Как увидеть знаки структурных констант?}, Алгебра и анализ {\bf 19} (2007), №4, 34--68. 

\item Вавилов Н. А., Лузгарев А. Ю., {\it Нормализатор группы Шевалле типа $\E_7$}, Алгебра и анализ {\bf 27} (2015), №6, 57--88.

\item Вавилов Н. А., Лузгарев А. Ю., Певзнер И. М., {\it Группа
Шевалле типа $\E_6$ в $27$-мерном представлении}, Зап. научн.
семин. ПОМИ {\bf 338} (2006), 5--68.

\item Вавилов Н. А., Певзнер И. М., {\it Тройки длинных корневых
подгрупп}, Зап. научн. семин. ПОМИ {\bf 343} (2007), 54--83.

\item Вавилов Н. А., Семенов А. А., {\it Длинные корневые торы в
группах Шевалле}, Алгебра и анализ {\bf 24} (2012), №3, 22--83.

\item Лузгарев А. Ю., Певзнер И. М., {\it Некоторые факты из жизни
$\GL(5,\Z)$}, Зап. научн. семин. ПОМИ {\bf 305} (2003), 153--163.

\item О'Мира О., {\it Лекции о линейных группах}, Автоморфизмы
классических групп, Мир, М., 1976, с. 57--167.

\item О'Мира О., {\it Лекции о симплектических группах}, Мир, М.,
1979.

\item Певзнер И. М., {\it Геометрия корневых элементов в группах
типа $\E_6$}, Алгебра и анализ {\bf 23} (2011), №3, 261--309.

\item Певзнер И. М., {\it Ширина групп типа ${\E}_{6}$
относительно множества корневых элементов}, {\rm I}, Алгебра и
анализ {\bf 23} (2011), №5, 155--198.

\item Певзнер И. М., {\it Ширина групп типа ${\E}_{6}$
относительно множества корневых элементов}, {\rm II}, Зап. научн.
семин. ПОМИ {\bf 386} (2011), 242--264.

\item Певзнер И. М., {\it Ширина группы $\GL(6,K)$ относительно
множества квазикорневых элементов}, Зап. научн. семин. ПОМИ {\bf
423} (2014), 183--204.

\item Певзнер И. М., {\it Ширина экстраспециального унипотентного
радикала относительно множества корневых элементов}, Зап. научн.
семин. ПОМИ {\bf 435} (2015), 168--177.

\item Певзнер И. М., {\it Существование корневой подгруппы, которую данный элемент переводит в противоположную}, Зап. научн. семин. ПОМИ {\bf 460} (2017), 190--202.

\item Певзнер И. М., {\it Орбиты векторов некоторых представлений. {\rm I}}, Зап. научн. семин. ПОМИ {\bf 484} (2019), 149--164.

\item Спрингер Т. А., {\it Линейные алгебраические группы},
Алгебраическая геометрия -- 4, Итоги науки и техн. Сер. Соврем.
проблемы мат. Фундам. направления {\bf 55}, ВИНИТИ, М., 1989, с.
5--136.

\item Стейнберг Р., {\it Лекции о группах Шевалле}, Мир, М., 1975.

\item Хамфри Дж., {\it Линейные алгебраические группы}, Наука, М.,
1980.

\item Хамфри Дж., {\it Введение в теорию алгебр Ли и их
представлений}, МЦНМО, М., 2003.

\item Aschbacher M., {\it The 27-dimensional module for $E_6$. {\rm I}}, Invent. Math {\bf 89} (1987), no. 1, 159--195.

\item Aschbacher M., {\it The 27-dimensional module for $E_6$. {\rm II}}, J. London Math. Soc {\bf 37} (1988), 275--293.

\item Aschbacher M., {\it The 27-dimensional module for $E_6$. {\rm III}}, Trans. Amer. Math. Soc. {\bf 321} (1990), 45--84.

\item Aschbacher M., {\it The 27-dimensional module for $E_6$. {\rm IV}}, J. Algebra {\bf 131} (1990), 23--39.

\item Aschbacher M., {\it Some multilinear forms with large isometry groups}, Geom.\ Dedicata {\bf 25} (1988), no. 1--3, 417--465.

\item Aschbacher M., {\it The geometry of trilinear forms}, Finite Geometries, Buildings and Related
topics, Oxford: Oxford Univ. Press (1990), 75--84.

\item Cooperstein B. N., {\it The fifty-six-dimensional module for $E_7$. I. A four form for $E_7$}, J. Algebra {\bf 173} (1995), no. 2, 361--389.

\item Krutelevich S., {\it Jordan algebras, exceptional groups, and Bhargava composition}, J. Algebra {\bf 314} (2007), no. 2, 924--977.

\item Springer T. A., {\it Linear algebraic groups}, Progress in
Mathematics {\bf 9}, Birkh\"auser Boston Inc., Boston, 1998.

\item Tits J., {\it Sur les constantes de structure et le th{\' e}or{\` e}me d'existence des alg{\` e}bres de Lie semi-simples}, Inst. Hautes {\' E}tudes Sci. Publ. Math. No. 31 (1966), 21--58.

\item Vavilov N. A., {\it A third look at weight diagrams}, Rend. Sem. Mat. Univ. Padova {\bf 104} (2000), 201--250.

\item Vavilov N. A., Plotkin E. B., {\it Chevalley groups over commutative rings. I. Elementary calculations}, Acta Applicandae Math. {\bf 45} (1996), 73--115.

\end{enumerate}

\end{document}